\documentclass{amsart}
\usepackage{amssymb,mathrsfs}%MnSymbol}%,skak}
\def\bR {\mathbf{R}}

\def\cD {\mathcal{D}}

\def\cT {\mathcal{T}}

\def\a {{\alpha}}
\def\b {{\beta}}
\def\g {{\gamma}}

\def\eps {{\epsilon}}
\def\th {{\theta}}

\def\ka {{\kappa}}

\def\L {{\Lambda}}
\def\si {{\sigma}}

\def\om {{\omega}}
\def\Om {{\Omega}}

\def\d {{\partial}}
\def\grad {{\nabla}}

\def\rstr {{\big |}}
\def\indc {{\bf 1}}

\def\wto {{\rightharpoonup}}

\newcommand{\Div}{\operatorname{div}}

\newcommand{\Supp}{\operatorname{supp}}

\newcommand{\Osc}{\operatorname{osc}}

\def\wto {{\rightharpoonup}}
\def\wtost{\mathop{\wto}^{*\,\,}}

\newcommand{\ba}{\begin{aligned}}
\newcommand{\ea}{\end{aligned}}

\newcommand{\be}{\begin{equation}}
\newcommand{\ee}{\end{equation}}

\newcommand{\lb}{\label}

\newtheorem{Thm}{Theorem}[section]

\newtheorem{Lem}[Thm]{Lemma}

\begin{document}

\title[H\"older regularity for hypoelliptic kinetic equations]{H\"older regularity\\ for hypoelliptic kinetic equations\\ with rough diffusion coefficients}

\author[F. Golse]{Fran\c cois Golse}
\address[F.G.]{Ecole polytechnique, CMLS, 91128 Palaiseau Cedex, France}
\email{francois.golse@polytechnique.edu}

\author[A. Vasseur]{Alexis Vasseur}
\address[A.V.]{Department of Mathematics, University of Texas at Austin, 1 University Station - C1200, Austin, TX 78712-0257, USA}
\email{vasseur@math.utexas.edu}
\bibliographystyle{plain}

\date{\today}

\subjclass{35K65,35B65,35Q84}
 
\keywords{Hypoellipticity, Kinetic equations, Regularity,  Fokker-Planck equation, DeGiorgi method}

\thanks{
\textbf{Acknowledgment.} The work of A. F. Vasseur was partially supported by the NSF Grant DMS 1209420, and by a visiting professorship at Ecole polytechnique. Both authors thank L. Silvestre for his comments on a first version of this paper.}

\begin{abstract}
This paper is dedicated to the application of the DeGiorgi-Nash-Moser regularity theory to the  kinetic Fokker-Planck equation. This equation is hypoelliptic. It is parabolic only in the velocity variable, while the Liouville transport operator has a mixing effect in the position/velocity phase space.  The mixing effect is incorporated in the classical DeGiorgi method via the averaging lemmas. The result can be seen as a H\"older regularity version of the classical averaging lemmas. 
\end{abstract}

\maketitle 

%%%%%%%%%%%%%%%%%%%%%%%%%%%%%%%%%%%%%%%%%%%%%%%%%%%%%%%%%%%%%%%%%%%%%%%%%%%%%%%%%%%%%%%%%%%%%%%%%%%%%%%%%%%%%%%%%%%%%%%%%

\section{The Fokker-Planck equation}

%%%%%%%%%%%%%%%%%%%%%%%%%%%%%%%%%%%%%%%%%%%%%%%%%%%%%%%%%%%%%%%%%%%%%%%%%%%%%%%%%%%%%%%%%%%%%%%%%%%%%%%%%%%%%%%%%%%%%%%%%

This paper is dedicated to the application of the DeGiorgi method to hypoelliptic equations, with rough coefficients. DeGiorgi  introduced his technique \cite{DeG}  in 1957 to solve Hilbert's 19th problem. In this work, he proved the regularity of 
variational solutions to nonlinear elliptic problems.  Independently, Nash introduced a similar technique \cite{Nash} in 1958.  Subsequently, Moser provided a new formulation of the proof in \cite{Moser}. Those methods are now usually called 
DeGiorgi-Nash-Moser techniques.  The method  has been extended to degenerate cases, like the $p$-Laplacian, first in the elliptic case by Ladyzhenskaya and Uralt'seva \cite{L1}.  The degenerate parabolic cases were covered later by 
DiBenedetto \cite{Di1} (see also DiBenedetto, Gianazza and Vespri \cite{L2,L3,L4}). More recently, the method has been extended to integral operators, such as fractional diffusion, in \cite{CV1, CV2} --- see also the work of Kassmann \cite{K} 
and of Kassmann and Felsinger \cite{K1}. Further application to fluid mechanics can be found in \cite{V2,V3,V4}.

Let  $A\equiv A(t,x,v)$ be an $M_N(\bR)$-valued measurable map on $\bR\times\bR^N\times\bR^N$ such that
\be\lb{CondA}
\frac1\L I\le A(t,x,v)=A(t,x,v)^T\le\L I
\ee
for some $\L>1$. Given $T\ge 0$, consider the Fokker-Planck equation with unknown $f\equiv f(t,x,v)\in\bR$
\be\lb{FPEq}
(\d_t+v\cdot\grad_x)f(t,x,v)=\Div_v(A(t,x,v)\grad_vf(t,x,v))+g(t,x,v)
\ee
for $x,v\in\bR^N$ and $t>-T$, where $g\equiv g(t,x,v)$ is given.

Assuming that $g\in L^2_{loc}([-T,\infty); L^2(\bR^N\times\bR^N))$, it is natural to seek $f$ so that
\be\lb{Spacf}
f\in C((-T,\infty);L^2(\bR^N\times\bR^N)\quad\hbox{ and }\grad_vf\in L^2_{loc}([-T,\infty); L^2(\bR^N\times\bR^N))\,,
\ee
in view of the following energy inequality:
\be\lb{Energ}
\ba
\tfrac12\|f(t,\cdot,\cdot)\|^2_{L^2(\bR^N\times\bR^N)}+\frac1\L\int_{t_0}^t\|\grad_vf(s,\cdot,\cdot)\|^2_{L^2(\bR^N\times\bR^N)}ds&
\\
\le\tfrac12\|f(t_0,\cdot,\cdot)\|^2_{L^2(\bR^N\times\bR^N)}+\int_{t_0}^t\|g(s,\cdot,\cdot)\|_{L^2(\bR^N\times\bR^N)}\|f(s,\cdot,\cdot)\|_{L^2(\bR^N\times\bR^N)}ds&
\\
\le\tfrac12\|f(t_0,\cdot,\cdot)\|^2_{L^2(\bR^N\times\bR^N)}+\tfrac12\|g\|^2_{L^2((-T,\infty)\times\bR^N\times\bR^N)}&
\\
+\tfrac12\int_{t_0}^t\|f(s,\cdot,\cdot)\|^2_{L^2(\bR^N\times\bR^N)}ds&\,.
\ea
\ee
Applying Gronwall's inequality shows that leads therefore to the following bound on the solution of the Cauchy problem for the Fokker-Planck equation with initial data $f\rstr_{t=t_0}\in L^2(\bR^N\times\bR^N)$:
$$
\ba
\|f(t,\cdot,\cdot)\|^2_{L^2(\bR^N\times\bR^N)}&+\frac2\L\int_{-T}^{\tau}\|\grad_vf(s,\cdot,\cdot)\|^2_{L^2(\bR^N\times\bR^N)}ds
\\
\le&\left(\|f(t_0,\cdot,\cdot)\|^2_{L^2(\bR^N\times\bR^N)}+\|g\|^2_{L^2((-T,\tau)\times\bR^N\times\bR^N)}\right)e^{T+\tau}
\ea
$$
for each $\tau>0$ and each $t\in(-T,\tau)$. This bound involves only the $L^2$ bounds on the data $f\rstr_{t=t_0}$ and $g$.

This paper is organized as follows. Section 2 establishes a local $L^\infty$ bound for a certain class of weak solutions of the Fokker-Planck equation. The local H\"older regularity of these solutions is proved in section 3. As in the application
of the DeGiorgi method to parabolic equations, these two steps involve rather different arguments. The main result in the present paper is Theorem \ref{T-DG2}, at the beginning of section 3. Yet, the local $L^\infty$ bound obtained in section 2 
is of independent interest and is a important ingredient in the proof of local H\"older regularity in section 3. For that reason, we have stated this local $L^\infty$ bound separately as Theorem \ref{T-DG1} at the beginning of section 2.

The arguments used in this paper follow the general strategy used by DeGiorgi, with significant differences, due to the hypoelliptic nature of the Fokker-Planck equation. Earlier results based on Moser's method are reported in the literature:
see \cite{PascuPoli,WZ}. The method used in the present paper is especially adapted to kinetic models.

Shortly after completing our proof of local H\"older regularity (Theorem \ref{T-DG2}), we learned of an independent approach of this problem by Imbert and Mouhot \cite{IM}. The main difference between \cite{IM} and our own work is that
Imbert and Mouhot follow Moser's approach, while we follow DeGiorgi's argument. 

%%%%%%%%%%%%%%%%%%%%%%%%%%%%%%%%%%%%%%%%%%%%%%%%%%%%%%%%%%%%%%%%%%%%%%%%%%%%%%%%%%%%%%%%%%%%%%%%%%%%%%%%%%%%%%%%%%%%%%%%%

\section{The Local $L^\infty$ Estimate}
%%%%%%%%%%%%%%%%%%%%%%%%%%%%%%%%%%%%%%%%%%%%%%%%%%%%%%%%%%%%%%%%%%%%%%%%%%%%%%%%%%%%%%%%%%%%%%%%%%%%%%%%%%%%%%%%%%%%%%%%%

Assume henceforth that $T>\tfrac32$. All solutions $f$ of the Fokker-Planck equation considered here are assumed to satisfy (\ref{Spacf}) and are renormalized in the sense that, for each $\chi\in C^2(\bR)$ satisfying $\chi(z)=O(z^2)$ as 
$|z|\to\infty$, one has
\be\lb{Renorm}
(\d_t+v\cdot\grad_x)\chi(f)=\Div_v(A\grad_v\chi(f))-\chi''(f)A:(\grad_vf)^{\otimes 2}+g\chi'(f)
\ee
in the sense of distributions on $(-T,\infty)\times\bR^N\times\bR^N$. 

\smallskip
\noindent
\textbf{Notation:} for each $r>0$, we set 
$$
Q[r]:=(-r,0)\times B(0,r)\times B(0,r)\,.
$$

\smallskip
The goal of this section is to prove the following local $L^\infty$ bound. This is the first important step in the DeGiorgi method.

\begin{Thm}\lb{T-DG1}
For each $\L>1$, each $\g>0$, and each $q>12N+6$, there exists $\ka\equiv\ka[N,M,\L,\g,q]\in(0,1)$ satisfying the following property.

For each $M_N(\bR)$-valued measurable map $A$ on $\bR\times\bR^N\times\bR^N$ satisfying (\ref{CondA}), each $g\in L^q(Q[\tfrac32])$ such that
$$
\|g\|_{L^q(Q[\tfrac32])}\le\g\,,
$$
and each $f\in C_b((-\tfrac32,0);L^2(B(0,\tfrac32)^2))$, solution of the Fokker-Planck equation
$$
(\d_t+v\cdot\grad_x)f=\Div_v(A\grad_vf)+g\quad\hbox{ on }Q[\tfrac32]\,,
$$
the following implication is true:
$$
\int_{Q[\tfrac32]}f(t,x,v)_+^2dtdxdv<\ka\Rightarrow f\le\tfrac12\hbox{ a.e. on }Q[\tfrac12]\,.
$$
\end{Thm}

The proof of Theorem \ref{T-DG1} involves several steps, following more or less closely DeGiorgi's original strategy. We shall insist on those steps which significantly differ from DeGiorgi's classical argument.

\subsection{The Local Energy Inequality}
%%%%%%%%%%%%%%%%%%%%%%%%%%%%%%%%%%%%%%%%%%%%%%%%%%%%%%%%%%%%%%%%%%%%%%%%%%%%%%%%%%%%%%%%%%%%%%%%%%%%%%%%%%%%%%%%%%%%%%%%%

Since the solution $f$ of the Fokker-Planck equation considered here is renormalized, for each $\psi\in C^\infty_c((-T,\infty)\times\bR^N\times\bR^N)$, one has
$$
\ba
\int_{-T}^\infty\iint_{\bR^N\times\bR^N}A:(\grad_v\chi(f)\otimes\grad_v\psi+\psi\chi''(f)(\grad_vf)^{\otimes 2}) dxdvd\tau&
\\
=\int_{-T}^\infty\iint_{\bR^N\times\bR^N}\chi(f)(\d_t+v\cdot\grad_x)\psi dxdvd\tau+\int_{-T}^\infty\iint_{\bR^N\times\bR^N}g\chi'(f)\psi dxdvd\tau&\,.
\ea
$$
Since $f\in C_b((-T,\infty);L^2(\bR^N\times\bR^N)$, one can pick a sequence of smooth test functions $\psi$ converging to a test function of the form $\psi(\tau,x,v)=\indc_{s<\tau<t}\phi(x,v)$. For each $\phi\in C^\infty_c(\bR^N\times\bR^N)$
one finds in this way that
\be\lb{tsWeakRen}
\ba
\int_s^t\iint_{\bR^N\times\bR^N}A:(\grad_v\chi(f)\otimes\grad_v\phi+\phi\chi''(f)(\grad_vf)^{\otimes 2}) dxdvd\tau&
\\
=\iint_{\bR^N\times\bR^N}\phi\chi(f)(s,x,v)dxdv-\iint_{\bR^N\times\bR^N}\phi\chi(f)(t,x,v)dxdv&
\\
+\int_s^t\iint_{\bR^N\times\bR^N}\chi(f)v\cdot\grad_x\phi dxdvd\tau+\int_t^s\iint_{\bR^N\times\bR^N}g\chi'(f)\phi dxdvd\tau&\,.
\ea
\ee

Let $\chi(f)=\tfrac12(f-c)^2_+$ for some $c\in\bR$ and pick $\eta\in C^\infty_c(\bR^N)$. Choosing the test function $\phi$ of the form $\phi(x,v)=\eta(x)\eta(v)^2$, we observe that
$$
\ba
A:(\grad_v\chi(f)\otimes\grad_v\phi+\phi\chi''(f)(\grad_vf)^{\otimes 2})
\\
=\eta(x)A:(2\eta(v)(f-c)_+\grad_vf\otimes\grad_v\eta(v)+\eta(v)^2\indc_{f>c}(\grad_vf)^{\otimes 2})
\ea
$$
since 
$$
\chi'(f)=(f-c)_+\quad\hbox{ and }\chi''(f)=\indc_{f>c}\,.
$$
Hence
$$
\ba
A:(\grad_v\chi(f)\otimes\grad_v\phi+\phi\chi''(f)(\grad_vf)^{\otimes 2})&
\\
=\eta(x)A:(2(\eta(v)\grad_v(f-c)_+)\otimes((f-c)_+\grad_v\eta(v))+\eta(v)^2(\grad_v(f-c)_+)^{\otimes 2})&
\\
=\eta(x)A:(\grad_v(\eta(v)(f-c)_+))^{\otimes 2}-\eta(x)A:((f-c)^2_+\grad\eta(v))^{\otimes 2}&\,,
\ea
$$
since $A=A^T$. Inserting this identity in (\ref{tsWeakRen}) shows that
\be\lb{LocEnerg}
\ba
\tfrac12\iint_{\bR^N\times\bR^N}\eta(x)\eta(v)^2(f-c)^2_+(t,x,v)dxdv&
\\
+\frac1\L\int_s^t\iint_{\bR^N\times\bR^N}\eta(x)|\grad_v(\eta(v)(f-c)_+)|^2dxdvd\tau&
\\
\le\tfrac12\iint_{\bR^N\times\bR^N}\eta(x)\eta(v)^2(f-c)^2_+(s,x,v)dxdv&
\\
+\L\int_s^t\iint_{\bR^N\times\bR^N}\eta(x)(f-c)^2_+|\grad\eta(v)|^2dxdvd\tau&
\\
+\int_s^t\iint_{\bR^N\times\bR^N}\tfrac12\eta(v)^2(f-c)^2_+v\cdot\grad\eta(x) dxdvd\tau&
\\
+\int_s^t\iint_{\bR^N\times\bR^N}g(f-c)_+\eta(x)\eta(v)^2dxdvd\tau&\,.
\ea
\ee

\smallskip
\noindent
\textbf{Remark.} The function $\chi(z)=\tfrac12(z-c)_+^2$ is not $C^2$, but only $C^1$ with Lipschitz continuous derivative. Instead of arguing directly with $\chi$ as above, one should replace $\chi$ by a  smooth approximation 
$\chi_\eps$ and passes to the limit as the small parameter $\eps\to 0$.

\subsection{The Dyadic Truncation Procedure}
%%%%%%%%%%%%%%%%%%%%%%%%%%%%%%%%%%%%%%%%%%%%%%%%%%%%%%%%%%%%%%%%%%%%%%%%%%%%%%%%%%%%%%%%%%%%%%%%%%%%%%%%%%%%%%%%%%%%%%%%%

This step closely follows DeGiorgi's classical method. For each integer $k\ge -1$, we define
$$
T_k:=-\tfrac12(1+2^{-k})\,,\qquad R_k:=\tfrac12(1+2^{-k})
$$
and we set
$$
B_k:=B(0,R_k)\,,\qquad Q_k:=(T_k,0)\times B_k^2(=Q[R_k])\,.
$$
Pick $\eta_k\in C^\infty(\bR^N)$ such that $0\le\eta_k\le 1$, satisfying
$$
\eta_k\equiv 1\hbox{  on }\overline{B_k}\,,\quad\eta_k\equiv 0\hbox{  on }B_{k-1}^c\,,\quad\hbox{ and }\|\grad\eta_k\|_{L^\infty}\le 2^{k+2}\,.
$$
Finally, set
$$
C_k:=\tfrac12(1-2^{-k})\,,\qquad\hbox{ and }f_k=(f-C_k)_+\,.
$$
Write inequality (\ref{LocEnerg}) with $\eta=\eta_k$ and $c=C_k$, for each $s\in(T_{k-1},T_k)$: one has
$$
\ba
\tfrac12\iint_{\bR^N\times\bR^N}\eta_k(x)\eta_k(v)^2f_k^2(t,x,v)dxdv&
\\
+\frac1\L\int_{T_k}^t\iint_{\bR^N\times\bR^N}\eta_k(x)|\grad_v(\eta_k(v)f_k)|^2dxdvd\tau&
\\
\le\tfrac12\iint_{\bR^N\times\bR^N}\eta_k(x)\eta_k(v)^2f_k^2(s,x,v)dxdv&
\\
+\L\int_{T_{k-1}}^t\iint_{\bR^N\times\bR^N}\eta_k(x)f_k^2|\grad\eta_k(v)|^2dxdvd\tau&
\\
+\int_{T_{k-1}}^t\iint_{\bR^N\times\bR^N}\tfrac12\eta_k(v)^2f_k^2v\cdot\grad\eta_k(x) dxdvd\tau&
\\
+\int_{T_{k-1}}^t\iint_{\bR^N\times\bR^N}gf_k\eta_k(x)\eta_k(v)^2dxdvd\tau&\,.
\ea
$$
Averaging both sides of the inequality above in $s\in(T_{k-1},T_k)$ shows that
$$
\ba
\tfrac12\iint_{\bR^N\times\bR^N}\eta_k(x)\eta_k(v)^2f_k^2(t,x,v)dxdv&
\\
+\frac1\L\int_{T_k}^t\iint_{\bR^N\times\bR^N}\eta_k(x)|\grad_v(\eta_k(v)f_k)|^2dxdvd\tau&
\\
\le 2^k\int_{T_{k-1}}^{T_k}\iint_{\bR^N\times\bR^N}\eta_k(x)\eta_k(v)^2f_k^2(s,x,v)dxdvds&
\\
+\L\int_{T_{k-1}}^t\iint_{\bR^N\times\bR^N}\eta_k(x)f_k^2|\grad\eta_k(v)|^2dxdvd\tau&
\\
+\int_{T_{k-1}}^t\iint_{\bR^N\times\bR^N}\tfrac12\eta_k(v)^2f_k^2v\cdot\grad\eta_k(x) dxdvd\tau&
\\
+\int_{T_{k-1}}^t\iint_{\bR^N\times\bR^N}gf_k\eta_k(x)\eta_k(v)^2dxdvd\tau&\,.
\ea
$$
Set
$$
\ba
U_k:=\sup_{T_k\le t\le 0}\tfrac12\iint_{\bR^N\times\bR^N}\eta_k(x)\eta_k(v)^2f_k^2(t,x,v)dxdv&
\\
+\frac1\L\int_{T_k}^0\iint_{\bR^N\times\bR^N}\eta_k(x)|\grad_v(\eta_k(v)f_k)|^2dxdvd\tau&\,.
\ea
$$
By construction
\be\lb{UkDecr}
0\le U_k\le U_{k-1}\le\ldots\le U_1\le U_0\,.
\ee
Now
$$
\ba
2^k\int_{T_{k-1}}^{T_k}\iint_{\bR^N\times\bR^N}\eta_k(x)\eta_k(v)^2f_k^2(s,x,v)dxdvds&
\\
+\L\int_{T_{k-1}}^t\iint_{\bR^N\times\bR^N}\eta_k(x)f_k^2|\grad\eta_k(v)|^2dxdvd\tau&
\\
+\int_{T_{k-1}}^t\iint_{\bR^N\times\bR^N}\tfrac12\eta_k(v)^2f_k^2v\cdot\grad\eta_k(x) dxdvd\tau&
\\
\le
(2^k+\L 2^{2k+4}+\tfrac12\cdot R_{k-1}\cdot 2^{k+2})\int_{Q_{k-1}}f_k^2dxdvdt
\\
\le
2^{2k+3}(1+2\L)\int_{Q_{k-1}}f_k^2\indc_{f_k>0}dxdvdt\,,
\ea
$$
so that
\be\lb{Uk<L2}
U_k\le 2^{2k+3}(1+2\L)\int_{Q_{k-1}}f_k^2\indc_{f_k>0}dxdvdt+\int_{Q_{k-1}}|g||f_k|\indc_{f_k>0}dxdvdt\,.
\ee

\subsection{The Nonlinearization Procedure}
%%%%%%%%%%%%%%%%%%%%%%%%%%%%%%%%%%%%%%%%%%%%%%%%%%%%%%%%%%%%%%%%%%%%%%%%%%%%%%%%%%%%%%%%%%%%%%%%%%%%%%%%%%%%%%%%%%%%%%%%%

This step starts as in DeGiorgi's classical argument. By H\"older's inequality, for each $p>2$, 
$$
\int_{Q_{k-1}}f_k^2\indc_{f_k>0}dxdvdt\le\left(\int_{Q_{k-1}}f_k^pdxdvdt\right)^{\frac2p}|\{f_k>0\}\cap Q_{k-1}|^{1-\frac2p}\,,
$$
while, assuming that $p>2$ is such that $\frac1p<\frac12-\frac1q$,
$$
\int_{Q_{k-1}}|g||f_k|\indc_{f_k>0}dxdvdt\le\|g\|_{L^q}\left(\int_{Q_{k-1}}f_k^pdxdvdt\right)^{\frac1p}|\{f_k>0\}\cap Q_{k-1}|^{1-\frac1p-\frac1q}\,.
$$
Now, for each $k\ge 0$, 
$$
\{f_k>0\}=\{f_{k-1}>C_k-C_{k-1}\}=\{f_{k-1}>2^{-k-1}\}
$$
so that, by Bienaym\'e-Chebyshev's inequality
$$
|\{f_k>0\}\cap Q_{k-1}|\le 2^{2k+2}\int_{Q_{k-1}}f_{k-1}^2dxdvdt\le 2^{2k+2}|T_{k-1}|U_{k-1}\le 3\cdot 2^{2k+1}U_{k-1}\,.
$$
Hence
\be\lb{Uk<}
\ba
U_k\le 2^{2k+3}(1+2\L)\left(\int_{Q_{k-1}}f_k^pdxdvdt\right)^{\frac2p}(3\cdot 2^{2k+1}U_{k-1})^{1-\frac2p}&
\\
+\|g\|_{L^q}\left(\int_{Q_{k-1}}f_k^pdxdvdt\right)^{\frac1p}(3\cdot 2^{2k+1}U_{k-1})^{1-\frac1p-\frac1q}&
\\
\le 3(1+2\L)2^{4k+4}U_{k-1}^{1-\frac2p}\left(\int_{Q_{k-1}}f_k^pdxdvdt\right)^{\frac2p}&
\\
+3\g 2^{2k+1}\cdot U_{k-1}^{1-\frac1p-\frac1q}\left(\int_{Q_{k-1}}f_k^pdxdvdt\right)^{\frac1p}&\,.
\ea
\ee

If one had an inequality of the form
\be\lb{WantedNonlin<}
\left(\int_{Q_{k-1}}f_k^pdxdvdt\right)^{\frac2p}\le C_kU_{k-1}\,,
\ee
the right-hand side of the inequality above would be the sum of two powers of $U_{k-1}$ with exponents
$$
1+1-\frac2p>1\quad\hbox{ and }1+\frac12-\frac1p-\frac1q>1\,.
$$
In other words, the bound obtained in the present step would result in a \textit{nonlinear} estimate for the \textit{linear} Fokker-Planck equation. Obtaining a nonlinear estimate for the solution of a linear equation is the key of DeGiorgi's local 
$L^\infty$ bound. The wanted inequality will be obtained by a variant of the velocity averaging method, to be explained in detail below.

\subsection{A Barrier Function}
%%%%%%%%%%%%%%%%%%%%%%%%%%%%%%%%%%%%%%%%%%%%%%%%%%%%%%%%%%%%%%%%%%%%%%%%%%%%%%%%%%%%%%%%%%%%%%%%%%%%%%%%%%%%%%%%%%%%%%%%%

In fact, the velocity averaging method will not be applied to $f_k$ itself, but to a barrier function dominating $f_k$. Constructing this barrier function is precisely the purpose of the present section. Set $F_k(t,x,v):=f_k(t,x,v)\eta_k(x)\eta_k(v)^2$. 
Then, one has
$$
\ba
(\d_t+v\cdot\grad_x)F_k-\Div_v(A\grad_vF_k)=g\indc_{f>C_k}\eta_k(x)\eta_k(v)^2+f_k\eta_k(v)^2v\cdot\grad\eta_k(x)
\\
-2f_k\Div_v(\eta_k(x)\eta_k(v)A\grad\eta_k(v))-4\eta_k(x)\eta_k(v)A:\grad_vf_k\otimes\grad\eta_k(v)-\mu_k
\ea
$$
where $\mu_k$ is a positive Radon measure because the function $z\mapsto (z-C_k)_+$ is convex. 

Set
$$
\ba
S_k:=g\indc_{f>C_k}\eta_k(x)\eta_k(v)^2+f_k\eta_k(v)^2v\cdot\grad\eta_k(x)
\\
-2\Div_v(\eta_k(x)\eta_k(v)f_kA\grad\eta_k(v))-2\eta_k(x)\eta_k(v)A:\grad_vf_k\otimes\grad\eta_k(v)
\ea
$$
and let $G_k$ be the solution of the initial boundary value problem
\be\lb{IBVPGk}
\left\{
\ba
{}&(\d_t+v\cdot\grad_x)G_k-\Div_v(A\grad_vG_k)=S_k\hbox{ on }Q_{k-1}\,,
\\
&G_k(t,x,v)=0\hbox{   if }|v|=R_{k-1}\hbox{ or }|x|=R_{k-1}\hbox{ and }v\cdot x<0\,,
\\
&G_k(T_{k-1},x,v)=0\,.
\ea
\right.
\ee
Hence
$$
\left\{
\ba
{}&(\d_t+v\cdot\grad_x)(G_k-F_k)-\Div_v(A\grad_v(G_k-F_k))=\mu_k\ge 0\hbox{ on }Q_{k-1}\,,
\\
&(G_k-F_k)(t,x,v)=0\hbox{  if }|v|=R_{k-1}\hbox{ or }|x|=R_{k-1}\hbox{ and }v\cdot x<0\,,
\\
&(G_k-F_k)(T_{k-1},x,v)=0\,,
\ea
\right.
$$
so that, by the maximum principle,
\be\lb{F<G}
0\le F_k\le G_k\hbox{ a.e. on }Q_{k-1}\,.
\ee

\subsection{Using Velocity Averaging}
%%%%%%%%%%%%%%%%%%%%%%%%%%%%%%%%%%%%%%%%%%%%%%%%%%%%%%%%%%%%%%%%%%%%%%%%%%%%%%%%%%%%%%%%%%%%%%%%%%%%%%%%%%%%%%%%%%%%%%%%%

In DeGiorgi original method, the inequality (\ref{WantedNonlin<}) follows from the elliptic regularity estimate in the Sobolev space $H^1$ implied by the energy inequality. Together with Sobolev embedding, this leads to an exponent 
$p>2$ in (\ref{WantedNonlin<}). 

In the case of the Fokker-Planck equation considered here, the energy inequality (\ref{Energ}) gives $H^1$ regularity in the $v$ variables only, and not in $(t,x)$. A natural idea is to use the hypoelliptic nature of the Fokker-Planck equation
in order to obtain some amount of regularity in $(t,x)$. The lack of regularity of the diffusion coefficients, i.e. of the entries of the matrix $A$ forbids using the classical methods in H\"ormander's theorem \cite{Hor3}.

There is another strategy for obtaining regularity in hypoelliptic equations of Fokker-Planck type, which is based on the velocity averaging method for kinetic equations. Velocity averaging designates a special type of smoothing effect for 
solutions of the free transport equation 
$$
(\d_t+v\cdot\grad_x)f=S
$$
observed for the first time in \cite{Agosh,GPS} independently, later improved and generalized in \cite{GLPS,DPL}. This smoothing effect bears on averages of $f$ in the velocity variable $v$, i.e. on expressions of the form
$$
\int_{\bR^N}f(t,x,v)\phi(v)dv\,,
$$
say for $C^\infty_c$ test functions $\phi$. Of course, no smoothing on $f$ itself can be observed, since the transport operator is hyperbolic and propagates the singularities of the source term $S$. However, when $S$ is of the form
$$
S=\Div_v(A(t,x,v)\grad_vf)+g
$$
where $g$ is a given source term in $L^2$, the smoothing effect of velocity averaging can be combined with the $H^1$ regularity in the $v$ variable implied by the energy inequality (\ref{Energ}), in order to obtain some amount of
smoothing on the solution $f$ itself. A first observation of this type (at the level of a compactness argument) can be found in \cite{PLLCam}. More recently, Bouchut has obtained more quantitative results, in the form of Sobolev regularity
estimates \cite{Bouchut2002}. These estimates are one key ingredient in our proof.

By construction 
$$
\Supp(G_k)\subset Q_{k-1}\hbox{ and }\Supp(S_k)\subset Q_{k-1}\,,
$$
and one has
$$
(\d_t+v\cdot\grad_x)G_k-\Div_v(A\grad_vG_k)=S_k\hbox{ on }\bR\times\bR^N\times\bR^N\,.
$$

By definition
$$
S_k=S_{k,1}+\Div_vS_{k,2}
$$
with
$$
\ba
S_{k,1}:=&g\indc_{f>C_k}\eta_k(x)\eta_k(v)^2+f_k\eta_k(v)^2v\cdot\grad\eta_k(x)
\\
&-2\eta_k(x)\eta_k(v)A:\grad_vf_k\otimes\grad\eta_k(v)\,,
\\
S_{k,2}:=&-2\eta_k(x)\eta_k(v)f_kA\grad\eta_k(v)\,.
\ea
$$
Moreover
$$
\ba
\|S_{k,1}\|_{L^2(Q_{k-1})}&\le\|g\indc_{f>C_k}\|_{L^2(Q_{k-1})}
\\
&+2^{k+2}R_{k-1}\|f_k\|_{L^2(Q_{k-1})}+2^{k+3}\L\|\grad_vf_k\|_{L^2(Q_{k-1})}\,,
\\
\|S_{k,2}\|_{L^2(Q_{k-1})}&\le 2^{k+3}\L\|f_k\|_{L^2(Q_{k-1})}\,.
\ea
$$
Writing the energy inequality for (\ref{IBVPGk}), we find that
$$
\ba
\tfrac12\|G_k(t,\cdot,\cdot)\|^2_{L^2(B_{k-1}^2)}+\frac1\L\int_{T_{k-1}}^t|\grad_vG_k|^2dxdvds&
\\
\le\|S_{k,1}\|_{L^2(Q_{k-1})}\|G_k\|_{L^2(Q_{k-1})}+\|S_{k,2}\|_{L^2(Q_{k-1})}\|\grad_vG_k\|_{L^2(Q_{k-1})}&\,,
\ea
$$
so that
$$
\ba
\tfrac12\sup_{T_{k-1}\le t\le 0}\|G_k(t,\cdot,\cdot)\|^2_{L^2(B_{k-1}^2)}+\frac1\L\int_{T_{k-1}}^0|\grad_vG_k|^2dxdvds&
\\
\le |T_{k-1}|^{1/2}\|S_{k,1}\|_{L^2(Q_{k-1})}\sup_{T_{k-1}\le t\le 0}\|G_k(t,\cdot,\cdot)\|_{L^2(B_{k-1}^2)}&
\\
+\|S_{k,2}\|_{L^2(Q_{k-1})}\|\grad_vG_k\|_{L^2(Q_{k-1})}
\\
\le |T_{k-1}|\|S_{k,1}\|^2_{L^2(Q_{k-1})}+\tfrac14\sup_{T_{k-1}\le t\le 0}\|G_k(t,\cdot,\cdot)\|_{L^2(B_{k-1}^2)}&
\\
+\tfrac12\L\|S_{k,2}\|^2_{L^2(Q_{k-1})}+\frac1{2\L}\|\grad_vG_k\|_{L^2(Q_{k-1})}&\,.
\ea
$$
Hence
$$
\ba
\tfrac12\sup_{T_{k-1}\le t\le 0}\|G_k(t,\cdot,\cdot)\|^2_{L^2(B_{k-1}^2)}+\frac1{\L}\int_{T_{k-1}}^0|\grad_vG_k|^2dxdvds&
\\
\le 2|T_{k-1}|\|S_{k,1}\|^2_{L^2(Q_{k-1})}+\L\|S_{k,2}\|^2_{L^2(Q_{k-1})}&\,.
\ea
$$
In particular
\be\lb{BoundGk}
\ba
{}&\|G_k\|^2_{L^2(Q_{k-1})}\le 4T_{k-1}^2|\|S_{k,1}\|^2_{L^2(Q_{k-1})}+2\L|T_{k-1}|\|S_{k,2}\|^2_{L^2(Q_{k-1})}\,,
\\
&\|\grad_vG_k\|^2_{L^2(Q_{k-1})}\le \L|T_{k-1}|\|S_{k,1}\|^2_{L^2(Q_{k-1})}+\L^2\|S_{k,2}\|^2_{L^2(Q_{k-1})}\,.
\ea
\ee

The Fokker-Planck equation in (\ref{IBVPGk}) is recast as
$$
(\d_t+v\cdot\grad_x)G_k=S_{k,1}+\Div_v(S_{k,2}+A\grad_vG_k)\,,
$$
and we recall that $G_k$, $S_{k,1}$ and $S_{k,2}$ are supported in $Q_{k-1}$.

By velocity averaging, applying Theorem 1.3 of \cite{Bouchut2002} with $p=2$, $\ka=\Om=1$, $r=0$ and $m=1$, one finds that
$$
\ba
\|D_t^{1/3}G_k\|_{L^2}+\|D_x^{1/3}G_k\|_{L^2}&
\\
\le 
C_N\|G_k\|_{L^2}+C_N(1+R_k)^{2/3}\|D_vG_k\|^{2/3}_{L^2}(\|S_{k,2}\|^{1/3}_{L^2}+\L^{1/3}\|D_vG_k\|^{1/3}_{L^2})&
\\
+C_N(1+R_k)^{1/2}\|D_vG_k\|^{1/2}_{L^2}(\|S_{k,1}\|^{1/2}_{L^2}+\|S_{k,2}\|^{1/2}_{L^2}+\L^{1/2}\|D_vG_k\|^{1/2}_{L^2})&\,.
\ea
$$
On the other hand
$$
\|D_v^{1/3}G_k\|_{L^2}\le\|G_k\|^{2/3}_{L^2}\|D_vG_k\|^{1/3}_{L^2}\,.
$$

\subsection{Using the Sobolev Embedding Inequality}
%%%%%%%%%%%%%%%%%%%%%%%%%%%%%%%%%%%%%%%%%%%%%%%%%%%%%%%%%%%%%%%%%%%%%%%%%%%%%%%%%%%%%%%%%%%%%%%%%%%%%%%%%%%%%%%%%%%%%%%%%

With these inequalities, one can estimate 
$$
\|D_t^{1/3}G_k\|_{L^2}+\|D_x^{1/3}G_k\|_{L^2}+\|D_v^{1/3}G_k\|_{L^2}
$$
in terms of $U_{k-1}$, as follows. Indeed, 
$$
\|f_k\|^2_{L^2(Q_k)}\le 2|T_k|U_k\le 3U_k\,,\quad\hbox{ and }\|\grad_vf_k\|^2_{L^2(Q_k)}\le\L U_k\,.
$$
Moreover, by definition of $f_k:=(f-C_k)_+$, one has
$$
\|f_k\|^2_{L^2(Q_{k-1})}\le\|f_{k-1}\|^2_{L^2(Q_{k-1})}
$$
and 
$$
|\grad_vf_k|=\indc_{f>C_k}|\grad_vf|\le\indc_{f>C_{k-1}}|\grad_vf|=|\grad_vf_{k-1}|\,,
$$
so that
$$
\|f_k\|^2_{L^2(Q_{k-1})}\le 3U_{k-1}\,,\quad\hbox{ and }\|\grad_vf_k\|^2_{L^2(Q_{k-1})}\le\L U_{k-1}\,.
$$

Thus
$$
\ba
{}&\|S_{k,1}\|_{L^2}\le\|g\indc_{f>C_k}\|_{L^2(Q_{k-1})}+2^{k+1}(3\sqrt{3}+4\L^{3/2})U_{k-1}^{1/2}\,,
\\
&\|S_{k,2}\|_{L^2(Q_{k-1})}\le 2^{k+3}\sqrt{3}\L U_{k-1}^{1/2}\,.
\ea
$$
Besides, for $2<r\le q$, one has
$$
\|g\indc_{f>C_k}\|_{L^2(Q_{k-1})}\le\|g\|_{L^r(Q_{k-1})}|\{f>C_k\}\cap Q_{k-1}|^{\frac12-\frac1r}
$$
and we recall that
$$
|\{f_k>0\}\cap Q_{k-1}|\le 2^{2k+2}\int_{Q_{k-1}}f_{k-1}^2dxdvdt\le 2^{2k+2}|T_{k-1}|U_{k-1}\le 3\cdot 2^{2k+1}U_{k-1}\,.
$$
Therefore, for $2<r\le q$,
$$
\|S_{k,1}\|_{L^2}\le\|g\|_{L^r(Q_{k-1})}(3\cdot 2^{2k+1}U_{k-1})^{\frac12-\frac1r}+2^{k+1}(3\sqrt{3}+4\L^{3/2})U_{k-1}^{1/2}\,,
$$
so that
$$
\ba
{}&\|S_{k,1}\|^2_{L^2(Q_{k-1})}\le 6\|g\|^2_{L^r(Q_{k-1})}2^{(2k+1)(1-\frac2r)}U^{1-\frac2r}_{k-1}+2^{2k+4}(27+16\L^3)U_{k-1}\,.
\\
&\|S_{k,2}\|^2_{L^2(Q_{k-1})}\le 3\L^22^{2k+6}U_{k-1}\,.
\ea
$$

Inserting these bounds in the energy estimate (\ref{BoundGk}), we find that
$$
\ba
\|G_k\|^2_{L^2(Q_{k-1})}&\le 9\|S_{k,1}\|^2_{L^2(Q_{k-1})}+3\L\|S_{k,2}\|^2_{L^2(Q_{k-1})})
\\
&\le 54\|g\|_{L^r(Q_{k-1})}2^{(2k+1)(1-\frac2r)}U_{k-1}^{1-\frac2r}+9(27+24\L^3)2^{2k+4}U_{k-1}\,,
\ea
$$
while
$$
\ba
\|\grad_vG_k\|^2_{L^2(Q_{k-1})}&\le\tfrac32\L\|S_{k,1}\|^2_{L^2(Q_{k-1})}+\L^2\|S_{k,2}\|^2_{L^2(Q_{k-1})})
\\
&\le 9\L\|g\|^2_{L^r(Q_{k-1})}2^{(2k+1)(1-\frac2r)}U^{1-\frac2r}_{k-1}+3(27\L+24\L^4)2^{2k+3}U_{k-1}\,.
\ea
$$
In the inequalities above, one can use H\"older's inequality to estimate $\|g\|_{L^r(Q_{k-1})}$ as follows:
$$
\|g\|_{L^r(Q_{k-1})}\le\|g\|_{L^q(Q_{k-1})}|Q_{k-1}|^{\frac1r-\frac1q}\le|Q[\tfrac32]|^{1/2}\g\,.
$$

Summarizing, we have found that
$$
\|G_k\|^2_{L^2}+\|\grad_vG_k\|^2_{L^2}+\|S_{k,1}\|^2_{L^2}+\|S_{k,2}\|^2_{L^2}
\le
a^22^{2k}(U_{k-1}^{1-\frac2r}+U_{k-1})
$$
where $a\equiv a[\L,\g]>0$. With the velocity averaging estimate from the previous section, this implies that
$$
\|D_t^{1/3}G_k\|^2_{L^2}+\|D_x^{1/3}G_k\|^2_{L^2}+\|D_v^{1/3}G_k\|^2_{L^2}
\le
b^22^{2k}(U_{k-1}^{1-\frac2r}+U_{k-1})
$$
where $b=b[N,\L,\g]$ is given by
$$
b^2:=a^2(1+(6+5\L)C_N)\,.
$$
By Sobolev's embedding inequality, one has
$$
\|G_k\|^2_{L^p}\le K^2_Sb^22^{2k}(U_{k-1}^{1-\frac2r}+U_{k-1})\,,
$$
where $K_S$ is the Sobolev constant for the embedding 
$$
H^{1/3}(\bR\times\bR^N\times\bR^N)\subset L^p(\bR\times\bR^N\times\bR^N)\,,\quad\frac1p=\frac12-\frac1{6N+3}\,.
$$

By (\ref{F<G}), this implies that
$$
\|F_k\|^2_{L^p}\le K^2_Sb^22^{2k}(U_{k-1}^{1-\frac2r}+U_{k-1})\,,\quad \frac1p=\frac12-\frac1{6N+3}\,.
$$
Using (\ref{UkDecr}), we further simplify the inequality above to obtain
\be\lb{Fkp<}
\|F_k\|^2_{L^p}\le K^2_S(1+U_0^{\frac2r})b^22^{2k}U_{k-1}^{1-\frac2r}\,,\quad \frac1p=\frac12-\frac1{3N}
\ee
for each $k\ge 1$.

\subsection{The Induction Argument}
%%%%%%%%%%%%%%%%%%%%%%%%%%%%%%%%%%%%%%%%%%%%%%%%%%%%%%%%%%%%%%%%%%%%%%%%%%%%%%%%%%%%%%%%%%%%%%%%%%%%%%%%%%%%%%%%%%%%%%%%%

This last step closely follows DeGiorgi's classical argument. We return to (\ref{Uk<}), and observe that
$$
\ba
\int_{Q_{k-1}}f_k^pdxdvdt\le\int_{Q_{k-1}}f_{k-1}^pdxdvdt&=\int_{Q_{k-1}}F_{k-1}^pdxdvdt
\\
&\le\int_{Q_{k-2}}F_{k-1}^pdxdvdt=\|F_{k-1}\|^p_{L^p}\,.
\ea
$$
Therefore
$$
\ba
U_k&\le 3K^2_S(1+2\L)(1+U_0^{\frac2r})b^2\cdot 2^{6k+2}U_{k-1}^{1-\frac2p}U_{k-2}^{1-\frac2r}
\\
&+3K_S(1+U_0^{\frac2r})^{\frac12}b\|g\|_{L^q}\cdot 2^{3k}U_{k-1}^{1-\frac1p-\frac1q}U_{k-2}^{\frac12-\frac1r}
\\
&\le3K^2_S(1+2\L)(1+U_0^{\frac2r})b^2\cdot 2^{6k+2}U_{k-2}^{2-\frac2p-\frac2r}
\\
&+3K_S(1+U_0^{\frac2r})^{\frac12}b\|g\|_{L^q}\cdot 2^{3k}U_{k-2}^{1+\frac12-\frac1p-\frac1q-\frac1r}
\ea
$$
where the second inequality above follows from (\ref{UkDecr}).

With $p>2$ being the Sobolev exponent given by
$$
\frac1p=\frac12-\frac1{6N+3}\,,
$$
we choose $r=q>12N+6$ so that
$$
\ba
{}&2-\frac2p-\frac2r=2-1+\frac2{6N+3}-\frac2q=1+\frac2{6N+3}-\frac2q>1\,,
\\
&1+\frac12-\frac1p-\frac1q-\frac1r=1+\frac1{6N+3}-\frac2q>1\,.
\ea
$$

Besides, in view of (\ref{Uk<L2}), one has
$$
U_0\le8(1+2\L)\ka+\g\sqrt{\ka}|Q[\tfrac32]|^{\frac12-\frac1q}\le 8(1+2\L)+\g|Q[\tfrac32]|^{1/2}=:c[\L,\g]\,.
$$

Thus, setting
$$
\a:=1+\frac1{6N+3}-\frac2q>1\,,
$$
we obtain
$$
U_k\le C\cdot 2^{6k}U_{k-2}^\a\,,\quad k\ge 2\,,
$$
with
$$
C[N,\L,\g,q]:=12K^2_S(1+2\L)c^{\frac1{6N+3}}(1+c^{\frac2q})b^2+3K_S(1+c^{\frac2q})^{\frac12}b\g\,.
$$

With $V_k=U_{2k}$ and $\rho=2^{12}(1+C)$ --- notice that $\rho>1$ --- we recast this inequality as
$$
V_k\le\rho^kV_{k-1}^\a\,,\quad k\ge 1\,.
$$
Iterating, we find that
$$
\ba
V_k\le\rho^kV_{k-1}^\a\le\rho^{k+\a(k-1)}V_{k-2}^{\a^2}\le\rho^{k+\a(k-1)+\a^2(k-2)}V_{k-3}^{\a^3}
\\
\le\ldots\le\rho^{k+\a(k-1)+\a^2(k-2)+\ldots+\a^{k-1}}V_0^{\a^k}\,.
\ea
$$
Elementary computations show that
$$
\ba
k+\a(k-1)+\a^2(k-2)+\ldots+\a^{k-1}
\\
=k(1+\a+\ldots+\a^{k-1})-\a(1+2\a+\ldots+(k-1)\a^{k-2})
\\
=k\frac{\a^k-1}{\a-1}-\a\frac{d}{d\a}\frac{\a^k-1}{\a-1}
=k\frac{\a^k-1}{\a-1}-\a\left(\frac{k\a^{k-1}}{\a-1}-\frac{\a^k-1}{(\a-1)^2}\right)
\\
=\frac{\a(\a^k-1)-k(\a-1)}{(\a-1)^2}\le\frac{\a}{(\a-1)^2}\a^k
\ea
$$
so that
$$
V_k\le(\rho^{\frac{\a}{(\a-1)^2}}V_0)^{\a^k}\,,\hbox{ for each }k\ge 1
$$
since $\rho>1$. Choosing 
$$
V_0<\rho^{-\frac{\a}{(\a-1)^2}}
$$
implies that $V_k\to 0$ as $k\to+\infty$. By dominated convergence, this implies that
$$
\int_{Q[\frac12]}(f-\tfrac12)_+^2(t,x,v)dtdxdv=0\,,
$$
i.e. that
$$
f(t,x,v)\le\tfrac12\hbox{ for a.e. }(t,x,v)\in Q[\tfrac12]\,.
$$

Because of (\ref{Uk<L2}), 
$$
V_0\le 8(1+2\L)\ka+\g\sqrt{\ka}|Q[\tfrac32]|^{1/2}\le(8(1+2\L)+\g|Q[\tfrac32]|^{1/2})\sqrt\ka\,,
$$
so that one can choose
$$
\ka:=\min\left(\tfrac12,\frac{\rho^{-\frac{2\a}{(\a-1)^2}}}{(8(1+2\L)+\g|Q[\tfrac32]|^{1/2})^2}\right)\,.
$$

%%%%%%%%%%%%%%%%%%%%%%%%%%%%%%%%%%%%%%%%%%%%%%%%%%%%%%%%%%%%%%%%%%%%%%%%%%%%%%%%%%%%%%%%%%%%%%%%%%%%%%%%%%%%%%%%%%%%%%%%%

\section{The Local H\"older Continuity}
%%%%%%%%%%%%%%%%%%%%%%%%%%%%%%%%%%%%%%%%%%%%%%%%%%%%%%%%%%%%%%%%%%%%%%%%%%%%%%%%%%%%%%%%%%%%%%%%%%%%%%%%%%%%%%%%%%%%%%%%%

Our main result in this paper, i.e. the local smoothing effect at the level of H\"older continuity for the Fokker-Planck operator with rough diffusion matrix, is the following statement.

\begin{Thm}\lb{T-DG2}
Let $A$ be an $M_N(\bR)$-valued measurable function defined a.e. on $\bR\times\bR^N\times\bR^N$ satisfying (\ref{CondA}). Let $I$ be an open interval of $\bR$ and $\Om$ be an open subset of $\bR^N\times\bR^N$. Let $f\in C(I;L^2(\Om))$ 
and $g\in L^\infty(I\times\Om)$ satisfy the Fokker-Planck equation (\ref{FPEq}) on $I\times\Om$. Then there exists $\si>0$ such that, for each compact $K\subset I\times\Om$, one has $f\in C^{0,\si}(K)$.
\end{Thm}

Notice that, by Theorem \ref{T-DG1}, we already know that $f\in L^\infty_{loc}(I\times\Om)$.

As in the previous section, the proof of this result follows the general strategy of DeGiorgi's original argument, with significant differences.

\subsection{The Isoperimetric Argument}
%%%%%%%%%%%%%%%%%%%%%%%%%%%%%%%%%%%%%%%%%%%%%%%%%%%%%%%%%%%%%%%%%%%%%%%%%%%%%%%%%%%%%%%%%%%%%%%%%%%%%%%%%%%%%%%%%%%%%%%%%

An important step in the proof of regularity in DeGiorgi's method for elliptic equations is based on some kind of isoperimetric inequality (see the proof of Lemma II in \cite{DeG}). This isoperimetric inequality is a quantitative variant of the 
well-known fact that no $H^1$ function can have a jump discontinuity. More precisely, given an $H^1$ function $0\le u\le 1$ which takes the values $0$ and $1$ on sets of positive measure, DeGiorgi's isoperimetric inequality provides 
a lower bound on the measure of the set defined by the double inequality $0<u<1$. In the present section, we establish an analogue of DeGiorgi's isoperimetric inequality adapted to the free transport operator.

\smallskip
Set $\hat Q:=(-\tfrac32,-1]\times B(0,1)^2$. 

\begin{Lem}\lb{L-Isoperi}
Let $\L>1$ and $\eta>0$ be given, and let $\om\in(0,1-2^{-{1/N}})$. There exists $\th\in(0,\tfrac12)$ and $\a>0$ satisfying the following property.

Let $A\equiv A(t,x,v)$ be an $M_N(\bR)$-valued measurable map on $\bR\times\bR^N\times\bR^N$ satisfying (\ref{CondA}), and let $f,g$ be measurable functions on $\hat Q\cup Q[1]$ such that
$$
(\d_t+v\cdot\grad_x)f=\Div_v(A\grad_vf)+g\quad\hbox{ on }\hat Q\cup Q[1]\,,
$$
together with
$$
f\le 1\hbox{ and }|g|\le 1\hbox{ a.e. on }\hat Q\cup Q[1]
$$
and
$$
|\{f\le 0\}\cap\hat Q|\ge\tfrac12|\hat Q|\,.
$$
Then
$$
|\{f\ge 1-\th\}\cap Q[\om/2]|<\eta\,,
$$
or
$$
|\{0<f<1-\th\}\cap(\hat Q\cup Q[1])|\ge\a\,.
$$
\end{Lem}

While DeGiorgi's isoperimetric inequality is based on an explicit computation leading to a precise estimate with effective constants, the proof of Lemma \ref{L-Isoperi} is obtained by a compactness argument, so that the values of $\th$ and $\a$
are not known explicitly.

\begin{proof}
If the statement in Lemma \ref{L-Isoperi} was wrong, there would exist sequences $A_n$, $f_n$ and $g_n$ of measurable functions on $\hat Q\cup Q[1]$ satisfying (\ref{CondA}) and
\be\lb{FoPln}
(\d_t+v\cdot\grad_x)f_n=\Div_v(A\grad_vf_n)+g_n\quad\hbox{ on }\hat Q\cup Q[1]\,,
\ee
together with
$$
f_n\le 1\hbox{ and }|g_n|\le 1\hbox{ a.e. on }\hat Q\cup Q[1]\,,\quad|\{f_n\le 0\}\cap\hat Q|\ge\tfrac12|\hat Q|\,,
$$
and yet
$$
|\{f_n\ge 1-2^{-n}\}\cap Q[\om/2]|\ge\eta\,,\hbox{ while }|\{0<f_n<1-2^{-n}\}\cap(\hat Q\cup Q[1])|<2^{-n}\,.
$$
We shall see that this leads to a contradiction.

First, arguing as in (\ref{Uk<L2}), wee see that, for each $\rho\in(0,1)$, there exists $C_\rho>0$ such that 
$$
\|\grad_vf_n\|_{L^2((-\tfrac32\rho,0)\times B(0,\rho)^2)}\le C_\rho\,,\quad n\ge 2\,.
$$

By the Banach-Alaoglu theorem, one can assume that
$$
f_n\wtost f\hbox{ and }g_n\wtost g\hbox{ in }L^\infty(\hat Q\cup Q[1])\,,
$$
while
$$
A_n\grad_vf_n\wto h\hbox{ in }L^2((-\tfrac32\rho,0)\times B(0,\rho)^2)
$$
for each $\rho\in(0,1)$.

By the variant of hypoelliptic smoothing based on velocity averaging (Theorem 1.3 in \cite{Bouchut2002}), one has
$$
f_n\to f\hbox{ in }L^p_{loc}(\hat Q\cup Q[1])\quad\hbox{ for all }1<p<\infty\,.
$$

Hence
\be\lb{FoPl}
(\d_t+v\cdot\grad_x)f=\Div_vh+g\qquad\hbox{ on }\hat Q\cup Q[1]\,,
\ee
together with
$$
f\le 1\hbox{ and }|g|\le 1\hbox{ a.e. on }\hat Q\cup Q[1]\,,\quad |\{f\le 0\}\cap\hat Q|\ge\tfrac12|\hat Q|\,,
$$
and
$$
|\{f=1\}\cap Q[\om/2]|\ge\eta\,,\quad\hbox{ while }|\{0<f<1\}\cap(\hat Q\cup Q[1])|=0\,.
$$
One has also
$$
\|\grad_vf\|_{L^2(-\tfrac32\rho,0)\times B(0,\rho)^2}\le C_\rho\,.
$$

But
$$
(f_n)_+\to f_+\hbox{ in }L^p_{loc}(\hat Q\cup Q[1])\quad\hbox{ for all }1<p<\infty\,,
$$
and, since $f_n\le 1$ a.e. on $\hat Q\cup Q[1]$, we conclude that $f_+$ is an indicator function as it takes the values $0$ or $1$ a.e. on $\hat Q\cup Q[1]$. Besides, for each $\rho\in(0,1)$,
$$
\|\grad_vf_+\|_{L^2(-\tfrac32\rho,0)\times B(0,\rho)^2}=\|\indc_{f\ge 0}\grad_vf\|_{L^2(-\tfrac32\rho,0)\times B(0,\rho)^2}\le C_\rho\,.
$$
Thus, for a.e. $(t,x)\in(-\tfrac32\rho,0)\times B(0,\rho)$, the function $v\mapsto f_+(t,x,v)$ is a.e equal $0$ or $1$ and $v\mapsto\grad_vf_+(t,x,v)$ defines an element of $L^2(B(0,\rho))$. Hence, for a.e. $(t,x)\in(-\tfrac32\rho,0)\times B(0,\rho)$, 
the function $v\mapsto f_+(t,x,v)$ is either a.e. equal to $0$ or a.e. equal to $1$: see Sublemma on p. 8 in \cite{CV}. Therefore, $f_+$ is a.e. constant in the $v$ variable.

Likewise, for each $\rho\in(0,1)$, one has
$$
\|\grad_v(f_n)_+\|_{L^2(-\tfrac32\rho,0)\times B(0,\rho)^2}=\|\indc_{f_n\ge 0}\grad_vf_n\|_{L^2(-\tfrac32\rho,0)\times B(0,\rho)^2}\le C_\rho\,,
$$
so that
$$
A_n\grad_v(f_n)_+\wto k\quad\hbox{ in }L^2((-\tfrac32\rho,0)\times B(0,\rho)^2)
$$
for each $\rho\in(0,1)$. On the other hand, 
\be\lb{FoPlLim}
(\d_t+v\cdot\grad_x)f_+\le\Div_vk+1\qquad\hbox{ on }\hat Q\cup Q[1]\,,
\ee
by convexity of $z\mapsto z_+$.

Let us prove that $k=0$. For each $\phi\in C^\infty_c(\hat Q\cup Q[1])$, multiplying both sides of (\ref{FoPln}) by $\phi(f_n)_+$ and integrating in all variables, one finds that
$$
\ba
\int\phi A_n\grad_vf_n\cdot\grad_v(f_n)_+dtdxdv=\int\tfrac12(f_n)_+^2(\d_t+v\cdot\grad_x)\phi dtdxdv&
\\
-\int(f_n)_+A_n\grad_vf_n\cdot\grad_v\phi dtdxdv+\int g_n(f_n)_+\phi dtdxdv&
\\
\to\int\tfrac12f_+^2(\d_t+v\cdot\grad_x)\phi dtdxdv
-\int f_+h\cdot\grad_v\phi dtdxdv
\\
+\int gf_+\phi dtdxdv&\,.
\ea
$$
On the other hand, multiplying both sides of (\ref{FoPl}) by $f_+\phi$ and integrating in all variables, one find that
$$
\ba
\int\phi h\cdot\grad_vf_+dtdxdv=\int\tfrac12f_+^2(\d_t+v\cdot\grad_x)\phi dtdxdv-\int f_+h\cdot\grad_v\phi dtdxdv&
\\
+\int gf_+\phi dtdxdv&\,.
\ea
$$
Therefore
$$
A_n\grad_v(f_n)_+\cdot\grad_v(f_n)_+=A_n\grad_vf_n\cdot\grad_v(f_n)_+\to h\cdot\grad_vf_+\hbox{ in }\cD'(\hat Q\cup Q[1])\,.
$$

Observe that condition (\ref{CondA}) implies that, for each $0\le\phi\in C^\infty_c(\hat Q\cup Q[1])$,
$$
\int\phi|A_n\grad_v(f_n)_+|^2dtdxdv\le\L\int\phi A_n\grad_v(f_n)_+\cdot\grad_v(f_n)_+dtdxdv\,.
$$
Therefore, by convexity and weak convergence
$$
\int\phi|k|^2dtdxdv\le\L\int\phi h\cdot\grad_vf_+dtdxdv\,.
$$
Since $f_+$ is a.e. constant in the variable $v$, one has
$$
\grad_vf_+=0\quad\hbox{ a.e. on }\hat Q\cup Q[1]\,,
$$
so that
$$
k=0\quad\hbox{ a.e. on }\hat Q\cup Q[1]\,.
$$

Eventually, $f_+=\indc_P(t,x)$ for some measurable $P\subset(-\tfrac32,0)\times B(0,1)$, with
\be\lb{Transpf+}
(\d_t+v\cdot\grad_x)f_+\le 1\quad\hbox{ in }\cD'(\hat Q\cup Q[1])\,,
\ee
with
\be\lb{IneqPhQ}
|(P\times B(0,1))\cap\hat Q|<\tfrac12|\hat Q|\quad\hbox{ and }|(P\times B(0,1))\cap Q[\om/2]|\ge\eta\,.
\ee
Since $f$ is independent of $v$, the inequality (\ref{Transpf+}) holds in $\cD'((-\tfrac32,0)\times B(0,1))$ for each $v\in B(0,1)$.

Since $f_+$ is an indicator function, for a.e. $(t_0,x_0,v_0)\in(\hat Q\cup Q[1])$, the function $s\mapsto f(t_0+s,x_0+sv_0,v_0)$ has jump discontinuities. Since 
$$
\frac{d}{ds}f(t_0+s,x_0+sv_0,v_0)=(\d_tf+v_0\cdot\grad_xf)(t_0+s,x_0+sv_0,v_0)\le 1\,,
$$
one has in fact
$$
\frac{d}{ds}f(t_0+s,x_0+sv_0,v_0)=(\d_tf+v_0\cdot\grad_xf)(t_0+s,x_0+sv_0,v_0)\le 0\,.
$$

On the other hand, if
$$
-\tfrac32<t_0\le-1\,,\quad |x_0|<1-\om\,,\quad\hbox{ and }-\tfrac12\om<t<0\,,\quad|x|<\tfrac12\om\,,
$$
then there exists $s>0$ and $v_0\in B(0,1)$ such that $(t,x,v)=(t_0+s,x_0+sv_0,v_0)$. Indeed,
$$
t-t_0>-\tfrac12\om-(-1)=1-\tfrac12\om\,,\quad\hbox{ and }|v|=\frac{|x-x_0|}{t-t_0}\le\frac{|x|+|x_0|}{t-t_0}<\frac{\tfrac12\om+1-\om}{1-\tfrac12\om}=1\,.
$$
Therefore
\be\lb{indP<}
\indc_P(t,x)\!\le\!\indc_P(t_0,x_0)\hbox{ for a.e. }(t_0,x_0,t,x)\!\in\!(-\tfrac32,-1]\!\times\!B(0,\!1\!-\!\om)\!\times\!(-\tfrac{\om}2,0)\!\times\!B(0,\!\tfrac{\om}2).
\ee
Since $(1-\om)^N>\tfrac12$, one has
$$
|(-\tfrac32,-1]\times B(0,1-\om)\times B(0,1)|>\tfrac12|\hat Q|\,,
$$
so that
$$
|((-\tfrac32,-1]\times B(0,1-\om))\setminus P|>0\,.
$$
Otherwise, 
$$
\ba
\tfrac12|\hat Q|<|(-\tfrac32,-1]\times B(0,1-\om)\times B(0,1)|&
\\
=|((-\tfrac32,-1]\times B(0,1-\om)\times B(0,1))\cap(P\times B(0,1))|\le|\hat Q\cap(P\times B(0,1))|&\,,
\ea
$$ 
which would contradict the first inequality in (\ref{IneqPhQ}).

Choosing $(t_0,x_0)\in (-\tfrac32,-1]\times B(0,1-\om))\setminus P$ in (\ref{indP<}), we conclude that 
$$
\indc_{P}(t,x)=0\hbox{ for a.e. }(t,x)\in (-\tfrac{\om}2,0)\times B(0,\tfrac{\om}2)\,.
$$
In other words, one has
$$
|(P\times B(0,1))\cap Q[\tfrac{\om}2]|=0\,,
$$
which is the desired contradiction with (\ref{IneqPhQ}).
\end{proof}

\subsection{Zooming in the Fokker-Planck Equation}
%%%%%%%%%%%%%%%%%%%%%%%%%%%%%%%%%%%%%%%%%%%%%%%%%%%%%%%%%%%%%%%%%%%%%%%%%%%%%%%%%%%%%%%%%%%%%%%%%%%%%%%%%%%%%%%%%%%%%%%%%

As in the DeGiorgi original proof, the local H\"older regularity is obtained by controling the oscillation of solutions of the Fokker-Planck equation on a sequence of domains with shrinking diameter. This suggests of course using a zooming
procedure based on the scaling properties of the Fokker-Planck equation. This step follows the classical DeGiorgi argument rather closely.

For each $t_0\in\bR$, $x_0,v_0\in\bR^3$ and $\eps>0$, we define the transformation $\cT_\eps[t_0,x_0,v_0]$ by the following prescription:
$$
\cT_\eps[t_0,x_0,v_0]F(s,y,\xi):=F(t_0+\eps^2s,x_0+\eps^3y+\eps^2sv_0,v_0+\eps\xi)\,.
$$
An elementary computation shows that, if
$$
(\d_t+v\cdot\grad_x)F=\Div_v(A\grad_vF)+G
$$
then $f(s,y,\xi)=\cT_\eps[t_0,x_0,v_0]F(s,y,\xi)$ satisfies
$$
(\d_s+\xi\cdot\grad_y)f=\Div_\xi(a\grad_\xi f)+g
$$
with
$$
a(s,y,\xi):=\cT_\eps[t_0,x_0,v_0]A(s,y,\xi)\quad\hbox{ and }g(s,y,\xi)=\eps^2\cT_\eps[t_0,x_0,v_0]G(s,y,\xi)\,.
$$
Observe that $a$ satisfies the same assumption as $A$, with the same constant $\L>1$.

Here is a first application of the zooming transformation defined above. With $\om$ chosen as in the previous section, i.e. $0<\om<1-2^{-1/N}$, set $\eps=\om/3$ in the zooming transformation defined above, together with $t_0=0$ and 
$x_0=v_0=0$. Assuming that $F$ satisfies
$$
(\d_t+v\cdot\grad_x)F=\Div_v(A\grad_vF)+G\hbox{ on }(-\tfrac{\om^2}6,0)\times B(0,\tfrac{\om^3}{18})\times B(0,\tfrac{\om}2)\,,
$$
then $f(s,y,\xi)=\cT_{\om/3}[0,0,0]F(s,y,\xi)$ satisfies
$$
(\d_s+\xi\cdot\grad_y)f=\Div_\xi(a\grad_\xi f)+g\hbox{ on }Q[\tfrac32]
$$
with
$$
a(s,y,\xi):=\cT_{\om/3}[0,0,0]A(s,y,\xi)\quad\hbox{ and }g(s,y,\xi)=\tfrac{\om^2}9\cT_{\om/3}[0,0,0]G(s,y,\xi)\,.
$$
By Theorem \ref{T-DG1}, assuming that $|g|\le\tfrac{\om^2}9$ a.e. on $Q[\tfrac32]$, one has the implication
$$
\int_{Q[3/2]}f_+^2dsdyd\xi<\ka[N,\L,\tfrac{\om^2}9,\infty]\Rightarrow f\le\tfrac12\hbox{ a.e. on }Q[\tfrac12]\,.
$$
In terms of $F$ and $G$, we arrive at the following statement: assuming that $|G|\le 1$ a.e. on $Q[\tfrac{\om}2]$, 
$$
\int_{Q[\om/2]}F_+^2dtdxdv<(\tfrac{\om}3)^{4N+2}\ka[N,\L,\tfrac{\om^2}9,\infty]\Rightarrow F\le\tfrac12\hbox{ a.e. on }Q[\tfrac1{54}\om^3]\,.
$$

\subsection{Reduction of Oscillation}
%%%%%%%%%%%%%%%%%%%%%%%%%%%%%%%%%%%%%%%%%%%%%%%%%%%%%%%%%%%%%%%%%%%%%%%%%%%%%%%%%%%%%%%%%%%%%%%%%%%%%%%%%%%%%%%%%%%%%%%%%

The second key idea in DeGiorgi's method for proving local regularity is the following important observation, which mixes the scaling transformation and the isoperimetric argument.

\begin{Lem}\lb{L-OscillRed}
There exist $\b,\mu\in(0,1)$ satisfying the following property. For each pair $f,g$ of measurable functions defined a.e. on $\hat Q\cup Q[1]$ such that
$$
(\d_t+v\cdot\grad_x)f=\Div_v(A\grad_vf)+g\hbox{ on }\hat Q\cup Q[1]
$$
with
$$
|f|\le 1\hbox{  and }|g|\le\b\quad\hbox{ on }\hat Q\cup Q[1]\,,
$$
one has
$$
\Osc_{Q[\om^3/54]}f\le\mu\Osc_{Q[\om/2]}f\,.
$$
\end{Lem}

\begin{proof}
Pick $\om\in(0,1-2^{-1/N})$, and set 
$$
\eta:=(\tfrac{\om}3)^{4N+2}\ka[N,\L,\tfrac{\om^2}9,\infty]\,.
$$
Lemma \ref{L-Isoperi} provides us with $\th\in(0,\tfrac12)$ and $\a>0$. Choose then $\b$ small enough so that
$$
\ln\frac1\b\ge\left(\frac{\frac12|\hat Q|+|Q[1]|}{\a}+2\right)\ln\frac1\th\,.
$$

Since $|f|\le 1$ on $\hat Q\cup Q[1]$, one can assume without loss of generality that
$$
|\{f\le 0\}\cap\hat Q|\ge\tfrac12|\hat Q|\,.
$$
(If 
$$
|\{f\le 0\}\cap\hat Q|<\tfrac12|\hat Q|\,,
$$
we shall argue instead with $-f$ and $-g$ instead of $f$ and $g$ respectively.)

Consider the sequence of functions defined by induction as follows:
$$
f_k=\frac1\th(f_{k-1}-1)+1\,,\quad f_0=f\,.
$$
One easily check by induction that
$$
f_k\le f_{k-1}\le\ldots\le f_1\le f_0=f\le 1\hbox{ a.e. on }\hat Q\cup Q[1]\,,
$$
and that $f_k$ is a solution of the Fokker-Planck equation on $\hat Q\cup Q[1]$ with source term
$$
g_k:=\th^{-k}g\,.
$$
We shall consider only finitely many terms in this sequence, viz. those for which
$$
0\le k\le k^*:=\left[\frac{\frac12|\hat Q|+|Q[1]|}{\a}\right]+1\le\left[\frac{\ln\b}{\ln\th}\right]\,.
$$
(Notice that the third inequality above follows from the constraint on $\b$ imposed at the begining of this proof.)

First, one has
$$
\{f\le 0\}\subset\{f_1\le 0\}\le\ldots\le\{f_{k-1}\le 0\}\le\{f_k\le 0\}\,,
$$
so that
$$
|\{f_k\le 0\}\cap\hat Q|\ge|\{f\le 0\}\cap\hat Q|\ge\tfrac12|\hat Q|\,.
$$

On the other hand
$$
\{f_k\le 0\}=\{f_{k-1}\le 0\}\cup\{0<f_{k-1}\le 1-\th\}
$$
so that the sequence
$$
m_k:=|\{f_k\le 0\}\cap(\hat Q\cup Q[1])|
$$
satisfies
$$
\ba
m_k&=m_{k-1}+|\{0<f_{k-1}\le 1-\th\}\cap(\hat Q\cup Q[1])|
\\
&=m_0+\sum_{l=1}^{k}|\{0<f_{l-1}\le 1-\th\}\cap(\hat Q\cup Q[1])|\,.
\ea
$$

It is obviously impossible that
$$
|\{0<f_{l-1}\le 1-\th\}\cap(\hat Q\cup Q[1])|\ge\a\hbox{ for each }l=1,\ldots,k^*\,,
$$
Indeed, this would imply that
$$
\tfrac12|\hat Q|+k^*\a\le m_0+k^*\a\le m_{k^*}\le|\hat Q|+|Q[1]|\,,
$$
which is impossible by our choice of $k^*$.

Notice that, by our choice of $\b$, one has
$$
\th^{-k^*}\b\le 1\,,\quad\hbox{ so that }|\th^{-k}g|\le 1\hbox{ a.e. on }\hat Q\cup Q[1]\hbox{ for }k=0,\ldots,k^*\,.
$$

Applying Lemma \ref{L-Isoperi} shows that there exists $\hat k\in\{0,\ldots,k^*-1\}$ such that
$$
|\{f_{\hat k}\ge 1-\th\}\cap Q[\om/2]|<\eta\,.
$$
Then
$$
\ba
\int_{Q[\om/2]}(f_{\hat k+1})_+^2dtdxdv&=\int_{Q[\om/2]}f_{\hat k+1}^2\indc_{f_{\hat k}\ge 1-\th}dtdxdv
\\
&\le\int_{Q[\om/2]}\indc_{f_{\hat k}\ge 1-\th}dtdxdv<\eta\,,
\ea
$$
so that 
$$
f_{\hat k+1}\le\tfrac12<1-\th\hbox{ a.e. on }Q[\tfrac1{54}\om^3]\,.
$$
By definition
$$
f_k-1=\th^{-k}(f-1)
$$
so that
$$
f=1+\th^{\hat k+1}(f_{\hat k+1}-1)\le 1-\th^{\hat k+2}\hbox{ a.e. on }Q[\tfrac1{54}\om^3]\,.
$$
In particular
$$
\Osc_{Q[\om^3/54]}f\le(1-\tfrac12\th^{\hat k+2})\Osc_{Q[\om/2]}f\le\mu\Osc_{Q[\om/2]}f
$$
with
$$
0\le(1-\tfrac12\th^{\hat k+2})\le\mu:=1-\th^{k^*+3}<1\,.
$$
\end{proof}

\subsection{Implications on H\"older Continuity}
%%%%%%%%%%%%%%%%%%%%%%%%%%%%%%%%%%%%%%%%%%%%%%%%%%%%%%%%%%%%%%%%%%%%%%%%%%%%%%%%%%%%%%%%%%%%%%%%%%%%%%%%%%%%%%%%%%%%%%%%%

With the estimates gathered above, we conclude the proof of local H\"older continuity as follows.

Observe that, for $r>0$ and $\eps\in(0,1)$,
$$
\Osc_{Q[r]}T_\eps[0,0,0]f=\Osc_{\tilde Q_\eps[r]}f\le\Osc_{Q[\eps r]}f\,,
$$
where
$$
\tilde Q_\eps[r]:=(-r\eps^2,0]\times B(0,r\eps^3)\times B(0,r\eps)\,.
$$

Assume that $f$ and $g$ satisfy the assumptions of Lemma \ref{L-OscillRed}; then 
$$
\left(\frac{\om^2}{27}\right)^n\|T^n_{\om^2/27}[0,0,0]g\|_{L^\infty(\hat Q\cup Q[1])}\le\b
$$
for each $n\ge 0$. Therefore, Lemma \ref{L-OscillRed} implies that
$$
\ba
\Osc_{Q[\om^3/54]}T^n_{\om^2/27}[0,0,0]f&\le\mu\Osc_{Q[\om/2]}T^n_{\om^2/27}[0,0,0]f
\\
&\le\mu\Osc_{Q[\om^3/54]}T^{n-1}_{\om^2/27}[0,0,0]f
\\
&\le\mu^n\Osc_{Q[\om^3/54]}f\le 2\mu^n
\ea
$$
pour tout $n\ge 0$.

In particular
$$
(s,y,\xi)\in Q[\tfrac{\om^3}{54}]\Rightarrow |f((\tfrac{\om^2}{27})^{2n}s,(\tfrac{\om^2}{27})^{3n}y,(\tfrac{\om^2}{27})^n\xi)-f(0,0,0)|\le 2\mu^n
$$
for each $n\ge 0$. Therefore
$$
-\frac{\om^{6n+3}}{2\cdot 27^{3n+1}}<t\le 0\,,\,\,|x|,|v|\le\frac{\om^{6n+3}}{2\cdot27^{3n+1}}\Rightarrow|f(t,x,v)-f(0,0,0)|\le 2\mu^n\,.
$$
In other words
$$
\ba
|f(s,y,\xi)-f(0,0,0)|\le 2\exp\left(\left[\frac{\ln(\frac{54}{\om^3}\max(|s|,|y|,|\xi|))}{\ln\frac{\om^2}{27}}\right]\ln\mu\right)
\\
\le\frac{2}{\mu^2}\exp\left(\frac{\ln(\frac{2}{\om}\max(|s|,|y|,|\xi|))}{\ln\frac{\om^2}{27}}\ln\mu\right)
\\
=\frac{2}{\mu^2}\left(\frac{2}{\om}\right)^{\ln\mu/\ln\frac{\om^2}{27}}\max(|s|,|y|,|\xi|)^{\ln\mu/\ln\frac{\om^2}{27}}\,.
\ea
$$

If $f$ and $g$ belong to $L^\infty(\hat Q\cup Q_[1])$ and satisfy the Fokker-Planck equation without satisfying the assumptions of Lemma \ref{L-OscillRed} on $\|f\|_{L^\infty(\hat Q\cup Q_[1])}$ and $\|g\|_{L^\infty(\hat Q\cup Q_[1])}$, replacing 
$f$ and $g$ respectively with $f/L$ and $g/L$ with
$$
L=(1+\|f\|_{L^\infty(\hat Q\cup Q_[1])})\left(1+\frac1\b\|g\|_{L^\infty(\hat Q\cup Q_[1])}\right)\,,
$$
we conclude that
$$
|f(s,y,\xi)-f(0,0,0)|\le C\max(|s|,|y|,|\xi|)^\si
$$
with
$$
\si:=\ln\mu/\ln\frac{\om^2}{28}
$$
and
$$
C:=\frac{2}{\mu^2}\left(\frac{2}{\om}\right)^\si(1+\|f\|_{L^\infty(\hat Q\cup Q_[1])})\left(1+\frac1\b\|g\|_{L^\infty(\hat Q\cup Q_[1])}\right)\,.
$$

Assume finally that $F$ is a solution of the Fokker-Planck equation with source term $G$ on some open neighborhood $\Om$ in $\bR\times\bR^N\times\bR^N$ of the point $(t_0,x_0,v_0)$. Assume further that $F,G\in L^\infty(\Om)$. 
Then $\cT_1[t_0,x_0,v_0]F$ is a solution of the Fokker-Planck equation with diffusion matrix $\cT_1[t_0,x_0,v_0]A$ and source term $\cT_1[t_0,x_0,v_0]G$. Arguing as above with $f:=\cT_1[t_0,x_0,v_0]F$, and setting
$$
s=t-t_0\,,\quad\xi=v-v_0\,,\quad\hbox{ and }y=x-x_0-sv_0\,,
$$
we conclude that
$$
|F(t,x,v)-F(t_0,x_0,v_0)|\le C((1+|v_0|)|t-t_0|+|v-v_0|+|x-x_0|)^\si
$$
provided that
$$
0<t_0-t<\frac{\om^3}{54(1+|v_0|)}\,,\quad|x-x_0|<\frac{\om^3}{54(1+|v_0|)}\,,\quad|v-v_0|<\frac{\om^3}{54}\,.
$$
Since $t_0$, $x_0$ and $v_0$ are arbitrary, this proves that $F$ is locally H\"older continuous with exponent $\si$.

%%%%%%%%%%%%%%%%%%%%%%%%%%%%%%%%%%%%%%%%%%%%%%%%%%%%%%%%%%%%%%%%%%%%%%%%%%%%%%%%%%%%%%%%%%%%%%%%%%%%%%%%%%%%%%%%%%%%%%%%%

%\bibliography{FPHypoEll2}

%\bibliographystyle{plain}

%%%%%%%%%%%%%%%%%%%%%%%%%%%%%%%%%%%%%%%%%%%%%%%%%%%%%%%%%%%%%%%%%%%%%%%%%%%%%%%%%%%%%%%%%%%%%%%%%%%%%%%%%%%%%%%%%%%%%%%%%

\end{document}